\RequirePackage{fix-cm}
\documentclass[smallextended,envcountsect,envcountsame]{svjour3} 

\smartqed 

\usepackage{amsmath,amssymb,mathtools}
\usepackage{microtype}
\usepackage[hidelinks]{hyperref}
\newcommand{\doi}[1]{{\urlstyle{same}\url{https://doi.org/#1}}}

\AddToHook{env/proof/end}{\qed}

\newcommand{\bR}{\mathbb{R}}
\newcommand{\bZ}{\mathbb{Z}}
\newcommand{\ri}{\operatorname{ri}}
\newcommand{\relbd}{\operatorname{relbd}}
\newcommand{\bdry}{\operatorname{bdry}}
\newcommand{\interior}{\operatorname{int}}

\newcommand{\cl}{\operatorname{cl}}
\newcommand{\cone}{\operatorname{cone}}
\newcommand{\dom}{\operatorname{dom}}
\newcommand{\epi}{\operatorname{epi}}
\newcommand{\dist}{\operatorname{dist}}

\newcommand{\dimH}{\operatorname{dim}_{\mathrm H}}
\DeclareMathOperator*{\argmin}{arg\,min}
\DeclareMathOperator*{\argmax}{arg\,max}

\begin{document}

\title{Solution Set Dimensions in Convex Optimization and Primal--Polar Facial Signatures%
}
\subtitle{}
\titlerunning{Solution Set Dimensions and Facial Signatures} 

\author{Jiaming Ma \and Defeng Sun}
\authorrunning{Jiaming Ma, Defeng Sun} 

\institute{Jiaming Ma \at
Department of Applied Mathematics, The Hong Kong Polytechnic University\\
\email{jiaming.ma@connect.polyu.hk}
\and
Defeng Sun \at
Department of Applied Mathematics, The Hong Kong Polytechnic University\\
\email{defeng.sun@polyu.edu.hk}
}

\date{Received: date / Accepted: date}

\maketitle

\begin{abstract}
We study the possible dimensions of solution sets in linearly perturbed convex optimization problems and their behavior under convex duality. The dimensions of solution sets and inverse solution sets capture two complementary forms of nonuniqueness: that of optimal solutions for a given perturbation, and that of perturbations for which a given point is optimal. Although these two families are linked by Fenchel conjugacy, we show that their collections of dimensions can be prescribed independently for finite supercoercive convex functions. As a geometric consequence, we prove that the facial dimension signatures of a convex body and its polar can be prescribed independently. The connection between these two realization results is obtained by homogenizing the epigraphs of a function and its conjugate. In contrast, under simultaneous facial exposedness, we establish a dimensional obstruction relating the smallest positive proper face dimensions of the body and its polar. These results resolve two open questions posed by Roshchina and Tun\c{c}el.
\end{abstract}

\section{Introduction}

Let $f:\bR^m\to\bR\cup\{+\infty\}$ be a closed proper convex function. For a linear perturbation $u\in\bR^m$, consider the convex optimization problem
\begin{equation}
\label{eq:Pu-intro}
    \min_{x\in\bR^m}\; f(x)-\langle u,x\rangle.
\end{equation}
Its optimal solution set is
\[
    S_f(u):=\argmin_{x\in\bR^m}\{f(x)-\langle u,x\rangle\}.
\]
Linear perturbations and subdifferential mappings are classical objects in convex and variational analysis; see, for example, \cite{Rockafellar1970,RockafellarWets1998}. We write $g^*$ and $\partial g$ for the Fenchel conjugate and the subdifferential of a convex function $g$, respectively. By the Fenchel--Young inequality,
\[
    S_f(u)=\partial f^*(u),
    \qquad
    S_f^{-1}(x)=\partial f(x)=S_{f^*}(x).
\]
The geometry of subdifferential graphs has been studied in variational and nonsmooth analysis; see, for example, \cite{DrusvyatskiyIoffeLewis2012}, where dimensional properties of such graphs are investigated.

Here we ask a different elementary question concerning the geometry of optimal solution sets under linear perturbations. The dimension of $S_f(u)$ captures the extent of nonuniqueness of the optimal solutions for a perturbation $u$, while the dimension of $S_f^{-1}(x)$ captures the extent of nonuniqueness of the perturbations for which a point $x$ is optimal. What dimensions can occur in these two families as $u$ and $x$ vary, and are the two resulting collections subject to any compatibility conditions under convex duality?

Our first main result gives a complete answer: the two collections can be prescribed independently. More precisely, for any $I,J\subseteq\{0,\ldots,m\}$ containing $0$, we construct a convex function $f$ such that
\[
    \{\dim S_f(u):u\in\bR^m\}=I,
    \qquad
    \{\dim S_f^{-1}(x):x\in\bR^m\}=J.
\]

As an application of this realization result, we obtain a corresponding result for facial dimensions under polarity. Define the facial dimension signature of a convex set $C$ by
\[
    \sigma(C):=\{\dim F:F \text{ is a face of } C,\ F\ne\emptyset\}.
\]
Roshchina and Tun\c{c}el \cite{RoshchinaTuncel2025} proved that every finite set of nonnegative integers can be realized as the facial dimension signature of a facially exposed convex set. They further asked whether two prescribed facial dimension signatures can be realized simultaneously by a compact convex set and its polar. We answer this question affirmatively: for any $I,J\subseteq\{0,\ldots,n\}$ with $\{0,n\}\subseteq I\cap J$, there exists a convex body $C\subseteq\bR^n$ with $0\in\interior C$ such that
\[
    \sigma(C)=I,\qquad
    \sigma(C^\circ)=J.
\]
Thus, the facial dimension signatures of a convex body and its polar can be prescribed independently, which answers the second open question in \cite[Section~4.4]{RoshchinaTuncel2025}.

Additional restrictions arise under facial exposedness. Suppose that $n\ge3$ and that $C$ and $C^\circ$ each have a positive-dimensional proper face. We prove that if both $C$ and $C^\circ$ are facially exposed, then
\[
    \min\bigl(\sigma(C)\setminus\{0\}\bigr)+\min\bigl(\sigma(C^\circ)\setminus\{0\}\bigr)\le n-1.
\]
As a consequence, we give a negative answer to the first open question in \cite[Section~4.4]{RoshchinaTuncel2025}.

The paper is organized as follows. In Section~\ref{sec:prelim}, we collect several geometric facts used throughout the paper. In Section~\ref{sec:realization}, we prove the simultaneous realization result for linearly perturbed convex optimization problems and derive the corresponding primal--polar realization theorem. In Section~\ref{sec:obstruction}, we establish the dimensional obstruction under facial exposedness.

\section{Preliminaries}
\label{sec:prelim}
\begin{sloppypar}
All spaces are finite-dimensional Euclidean spaces. Let $\interior S$, $\cl S$, and $\bdry S:=\cl S\setminus\interior S$ denote the interior, closure, and boundary of a set $S$, respectively. For a convex set $C$, we use $\dim C$, $\ri C$, $\relbd C$ for its dimension, relative interior, and relative boundary, respectively. A compact convex set with nonempty interior is called a convex body.
\end{sloppypar}

A convex subset $F\subseteq C$ is a \emph{face} of $C$, denoted by $F\unlhd C$, if whenever an open line segment $(x,y)$ with $x,y\in C$ intersects $F$, then $x,y\in F$. The empty set and $C$ itself are faces. All other faces are called proper. A face $F\unlhd C$ is \emph{exposed} if
\[
    F=\argmax_{x\in C}\langle u,x\rangle
\]
for some nonzero $u$. A convex set is \emph{facially exposed} if all its proper faces are exposed. For each $x\in C$, there is a unique face $F\unlhd C$ such that $x\in\ri F$, called the \emph{minimal face} containing $x$. Equivalently, it is the intersection of all faces of $C$ containing $x$. Following \cite{RoshchinaTuncel2025}, the \emph{facial dimension signature} of $C$ is defined by
\[
    \sigma(C):=\{\dim F:F\unlhd C,\ F\ne\emptyset\}.
\]

For a closed proper convex function $f:\bR^m\to\bR\cup\{+\infty\}$, its Fenchel conjugate is
\[
    f^*(u):=\sup_{x\in\bR^m}\{\langle u,x\rangle-f(x)\}.
\]
We use the standard notation $\dom f$ and $\epi f$ for its effective domain and epigraph. A convex function $f$ is \emph{supercoercive} if
\[
    \lim_{\|x\|\to\infty}\frac{f(x)}{\|x\|}=+\infty.
\]
For closed proper convex functions, $f$ is supercoercive if and only if $\dom f^*=\bR^m$; see \cite[Theorem~11.8]{RockafellarWets1998}. The subdifferential of $f$ at $x\in\dom f$ is
\[
    \partial f(x):=
    \{u\in\bR^m:f(y)\ge f(x)+\langle u,y-x\rangle,\ \forall y\in\bR^m\}.
\]
By \cite[Proposition~11.3]{RockafellarWets1998}, the equality condition of Fenchel--Young inequality gives
\begin{equation}
\label{eq:fenchel-subdiff}
    u\in\partial f(x)
    \quad\Longleftrightarrow\quad
    x\in\partial f^*(u)
    \quad\Longleftrightarrow\quad
    f(x)+f^*(u)=\langle u,x\rangle.
\end{equation}
If $f$ is finite and supercoercive, then for every $u\in\bR^m$,
\begin{equation}
\label{eq:Sf}
    S_f(u):=\argmin_{x\in\bR^m}\{f(x)-\langle u,x\rangle\}
    =\partial f^*(u)
\end{equation}
is a nonempty compact convex set. Denote the collection of possible dimensions of $S_f(u)$ by
\begin{equation}
\label{eq:Df}
    \mathcal D(f):=\{\dim S_f(u):u\in\bR^m\}.
\end{equation}

For a nonempty closed convex set $C\subseteq\bR^m$, its polar is
\[
    C^\circ:=\{u\in\bR^m:\langle u,x\rangle\le1,\ \forall x\in C\}.
\]
The cone generated by $C$ is denoted by
\[
    \cone C:=\{\lambda x:\lambda\ge0,\ x\in C\}.
\]
For a convex cone $K\subseteq\bR^m$, its polar cone is
\[
    K^\circ:=\{y\in\bR^m:\langle x,y\rangle\le0,\ \forall x\in K\}.
\]
The support function of $C$ is defined by
\[
    \delta_C^*(u):=\sup_{x\in C}\langle u,x\rangle,\qquad u\in\bR^m,
\]
which is the conjugate of the indicator function $\delta_C$. Its subdifferential is given by
\begin{equation}
\label{eq:support-subdiff}
    \partial\delta_C^*(u)=\argmax_{x\in C}\langle u,x\rangle.
\end{equation}
Thus, for $u\ne0$, $\partial\delta_C^*(u)$ is an exposed face of $C$ whenever it is nonempty, while $\partial\delta_C^*(0)=C$. Conversely, every exposed face of $C$ is of the form $\partial\delta_C^*(u)$ for some $u\in\bR^m$. In particular, if $C$ is facially exposed, then every nonempty face of $C$ can be represented in this form.

We next record the relation between the faces of epigraph and optimal solution sets.

\begin{lemma}
\label{lem:epi}
Let $f:\bR^m\to\bR$ be a convex function. Then
\[
    \sigma(\epi f)\setminus\{m+1\}
    =\bigcup_{u\in\bR^m}\sigma\bigl(\partial f^*(u)\bigr)
    =\bigcup_{u\in\bR^m}\sigma\bigl(S_f(u)\bigr).
\]
\end{lemma}

\begin{proof}
Fix $u\in\bR^m$. For any $(x,t)\in\epi f$,
\[
    \langle u,x\rangle-t
    \le \langle u,x\rangle-f(x)
    \le f^*(u).
\]
By \eqref{eq:fenchel-subdiff}, equality holds if and only if $x\in\partial f^*(u)$ and $t=f(x)$. Hence
\[
    E_u:=\argmax_{(x,t)\in\epi f}\{\langle u,x\rangle-t\}
    =\{(x,f(x)):x\in\partial f^*(u)\}.
\]
Thus $E_u$ is an exposed face of $\epi f$ and is affinely isomorphic to $\partial f^*(u)$.

Since $\dom f=\bR^m$, every supporting hyperplane of $\epi f$ has a normal of the form $(u,-1)$ up to a positive multiple. Hence every exposed proper face of $\epi f$ is some $E_u$. Every proper face of a closed convex set is contained in an exposed face and is a face of that exposed face; see \cite[Section~18]{Rockafellar1970}. Therefore
\[
    \sigma(\epi f)\setminus\{m+1\}
    =\bigcup_{u\in\bR^m}\sigma(E_u)
    =\bigcup_{u\in\bR^m}\sigma\bigl(\partial f^*(u)\bigr)
    =\bigcup_{u\in\bR^m}\sigma\bigl(S_f(u)\bigr).
\]
The last equality follows from \eqref{eq:Sf}. 
\end{proof}

We will use the following consequence of the construction in \cite[Theorem~1.1]{RoshchinaTuncel2025}.

\begin{lemma}
\label{lem:AB}
Let $m\ge1$, and let $I,J\subseteq\{0,\ldots,m\}$ be nonempty sets containing $0$. There exist facially exposed compact convex sets $A,B\subseteq\bR^m$ such that
\[
    \sigma(A)=I,
    \qquad
    \sigma(B)=J,
\]
and $B$ can be chosen so that
\begin{equation}
\label{eq:sharp-b}
    \interior N_B(b)\ne\emptyset
\end{equation}
for some $b\in B$, where $N_B(b)$ denotes the normal cone to $B$ at $b$.
\end{lemma}

\begin{proof}
Let $d_A:=\max I$ and $d_B:=\max J$. By \cite[Theorem~1.1]{RoshchinaTuncel2025}, there exist facially exposed compact convex sets with facial dimension signatures $I$ and $J$ in $\bR^{d_A}$ and $\bR^{d_B}$, respectively.

It remains to arrange \eqref{eq:sharp-b} for $B$. If $d_B\le1$, this is immediate. Suppose that $d_B\ge2$. In the construction of \cite[Theorem~1.1]{RoshchinaTuncel2025}, one may choose a boundary point near which the boundary agrees with that of a Euclidean ball. In suitable coordinates, write this part of the boundary as
\[
    \{(z,t):t=\phi(z):=-(1-\|z\|^2)^{1/2},\ \|z\|\le\varepsilon\}
\]
for some $\varepsilon\in(0,1)$. Set
\[
    \alpha:=\frac{1-(1-\varepsilon^2)^{1/2}}{\varepsilon},
\]
and replace this local part by
\[
    \left\{(z,t):
    t=\frac12\phi(z)+\frac12(\alpha\|z\|-1),\ \|z\|\le\varepsilon\right\}.
\]
The modified patch agrees with the original one at $\|z\|=\varepsilon$ and remains strictly convex. Hence the facial dimension signature is unchanged, and the resulting set remains facially exposed. Moreover, the term $\|z\|$ creates a corner at $(0,-1)$, so the normal cone there is full-dimensional in $\bR^{d_B}$.

Finally, embed the two sets affinely into $\bR^m$. Facial dimension signatures and facial exposedness are preserved under such embeddings, and a full-dimensional normal cone in $\bR^{d_B}$ yields a normal cone with nonempty interior in $\bR^m$. This proves \eqref{eq:sharp-b}.
\end{proof}

\section{Simultaneous realization under duality}
\label{sec:realization}

We first state the optimization result.

\begin{theorem}
\label{thm:optimization-realization}
Let $m\ge1$, and let $I,J\subseteq\{0,\ldots,m\}$ be nonempty sets containing $0$. Then there exists a finite supercoercive convex function $f:\bR^m\to\bR$ such that
\begin{equation}
\label{eq:D-realization}
    \mathcal D(f)=I,
    \qquad
    \mathcal D(f^*)=J.
\end{equation}
Moreover,
\begin{equation}
\label{eq:strong-realization}
    \bigcup_{u\in\bR^m}\sigma\bigl(S_f(u)\bigr)=I,
    \qquad
    \bigcup_{x\in\bR^m}\sigma\bigl(S_{f^*}(x)\bigr)=J.
\end{equation}
\end{theorem}

\begin{proof}
Choose facially exposed compact convex sets $A,B\subseteq\bR^m$ and $b\in B$ as in Lemma~\ref{lem:AB}. Let $D\subseteq\bR^m$ be a closed Euclidean ball such that
\[
    D\subseteq\interior N_B(b).
\]
After a translation and a positive scaling of $A$, we may assume that
\begin{equation}
\label{eq:A-in-D}
    A\subseteq\interior D.
\end{equation}
Define
\[
    q_A(x):=\frac12\dist(x,A)^2,\qquad q_D(x):=\frac12\dist(x,D)^2,
\]
and 
\begin{equation}
\label{eq:f-construction}
    f(x):=q_A(x)+q_D(x)+\delta_B^*(x).
\end{equation}
The support function $\delta_B^*$ is finite since $B$ is compact, and $q_D$ grows quadratically at infinity. Hence $f$ is finite and supercoercive, so $\dom f^*=\bR^m$. Moreover, since $\dom f=\bR^m$, $f^*$ is also supercoercive.

We first consider $S_f(u)=\partial f^*(u)$ for $u\in\bR^m$. Suppose that $\dim S_f(u)>0$. If $x,x'\in S_f(u)$ are distinct, then $f$ is affine on the closed segment $[x,x']$. Since all three terms in \eqref{eq:f-construction} are convex, each is affine on this segment. In particular, $q_D$ is affine on $[x,x']$. For the Euclidean ball $D$, this is possible only when
\[
    [x,x']\subseteq D.
\]
It follows that
\begin{equation}
\label{eq:positive-fiber-in-D}
    \dim S_f(u)>0\quad\Longrightarrow\quad S_f(u)\subseteq D.
\end{equation}

For $x\in D\subseteq\interior N_B(b)$, we have
$\partial q_D(x)=\{0\}$ and $\partial\delta_B^*(x)=\{b\}$, so
\[
    \partial f(x)=\partial q_A(x)+b.
\]
Hence, by \eqref{eq:fenchel-subdiff}, whenever $S_f(u)\subseteq D$,
\[
    S_f(u)=\partial q_A^*(u-b)\cap D.
\]
Since
\[
    q_A^*(z)=\frac12\|z\|^2+\delta_A^*(z),
\]
we obtain
\begin{equation}
\label{eq:fiber-formula-2}
    S_f(u)=\bigl(u-b+\partial\delta_A^*(u-b)\bigr)\cap D.
\end{equation}

We claim that
\begin{equation}
\label{eq:one-inclusion-I}
    \sigma(S_f(u))\subseteq\sigma(A)=I\qquad\forall u\in\bR^m.
\end{equation}
This is immediate if $\dim S_f(u)=0$, since $0\in I=\sigma(A)$. Otherwise, by \eqref{eq:positive-fiber-in-D}, $S_f(u)\subseteq D$. Let $G$ be a positive-dimensional face of $S_f(u)$ and choose $x\in\ri G$. Since $D$ is strictly convex, necessarily $x\in\interior D$. Set
\[
    H:=u-b+\partial\delta_A^*(u-b).
\]
By \eqref{eq:fiber-formula-2}, $S_f(u)=H\cap D$. Since $x\in\interior D$, the sets $S_f(u)$ and $H$ agree locally around $x$, and hence their minimal faces containing $x$ have the same dimension. Since $H$ is a translation of the face $\partial\delta_A^*(u-b)$ of $A$, every face of $H$ is a translation of a face of $A$. Thus $\dim G\in\sigma(A)$, and hence the claim \eqref{eq:one-inclusion-I} holds.

Conversely, let $d\in I=\sigma(A)$. Since $A$ is facially exposed, there exists $z\in\bR^m$ such that
\[
    \dim\partial\delta_A^*(z)=d,
\]
where $z=0$ may be used for the face $A$ itself. For every $x\in\bR^m$,
\[
    q_A(x)\ge0,\qquad q_D(x)\ge0,\qquad \delta_B^*(x)-\langle b,x\rangle\ge0.
\]
All three terms vanish on $A$, while $q_A(x)=0$ implies $x\in A$. Hence
\[
    S_f(b)=\argmin_{x\in\bR^m}\{q_A(x)+q_D(x)+(\delta_B^*(x)-\langle b,x\rangle)\}=A.
\]
Since $\dom f^*=\bR^m$, the mapping $\partial f^*$ is upper semicontinuous. By \eqref{eq:A-in-D}, there exists a neighborhood $U$ of $b$ such that
\[
    S_f(u)=\partial f^*(u)\subseteq\interior D,\qquad u\in U.
\]
By \eqref{eq:A-in-D}, the closed set $\partial\delta_A^*(z)$ is contained in $\interior D$. So, for all sufficiently small $\varepsilon>0$,
\[
    b+\varepsilon z\in U,\qquad \varepsilon z+\partial\delta_A^*(z)\subseteq\interior D.
\]
Using \eqref{eq:fiber-formula-2} and $\partial\delta_A^*(\varepsilon z)=\partial\delta_A^*(z)$, we obtain
\[
    S_f(b+\varepsilon z)=\varepsilon z+\partial\delta_A^*(z).
\]
Thus every $d\in I$ occurs as the dimension of $S_f(u)$ for some $u\in\bR^m$. Together with \eqref{eq:one-inclusion-I}, this gives
\[
    \mathcal D(f)=\bigcup_{u\in\bR^m}\sigma(S_f(u))=I.
\]

It remains to consider the conjugate problem. By \eqref{eq:fenchel-subdiff},
\[
    S_{f^*}(x)=\partial f(x).
\]
Since $q_A$ and $q_D$ are continuously differentiable,
\[
    \partial f(x)=\nabla q_A(x)+\nabla q_D(x)+\partial\delta_B^*(x).
\]
Thus $S_{f^*}(x)$ is a translation of the exposed face $\partial\delta_B^*(x)$ of $B$, and therefore
\[
    \sigma(S_{f^*}(x))\subseteq\sigma(B)=J.
\]
Conversely, since $B$ is facially exposed, every nonempty proper face of $B$ is of the form $\partial\delta_B^*(x)$ for some $x\in\bR^m$, while $B=\partial\delta_B^*(0)$. Consequently,
\[
    \mathcal D(f^*)=\bigcup_{x\in\bR^m}\sigma(S_{f^*}(x))=J.
\]
This proves \eqref{eq:D-realization} and \eqref{eq:strong-realization}.
\end{proof}


\begin{remark}
The identities in \eqref{eq:strong-realization} strengthen \eqref{eq:D-realization}: they prescribe not only the dimensions of the solution sets themselves, but also the dimensions of all their faces. This stronger form allows the realization result to pass to epigraphs through Lemma~\ref{lem:epi}.
\end{remark}

We now use Theorem~\ref{thm:optimization-realization} to derive the primal--polar realization result.

\begin{theorem}[Primal--polar realization]
\label{thm:polar-realization}
Let $I,J$ be nonempty subsets of $\bZ$ such that
\[
    \min I=\min J=0,
    \qquad
    \max I=\max J=n.
\]
Then there exists a convex body $C\subseteq\bR^n$ with $0\in\interior C$ such that
\[
    \sigma(C)=I,
    \qquad
    \sigma(C^\circ)=J.
\]
\end{theorem}

\begin{proof}
The case $n\le1$ is immediate, so assume $n\ge2$ and set $m:=n-1$. Apply Theorem~\ref{thm:optimization-realization} to
\[
    I_0:=I\setminus\{n\},
    \qquad
    J_0:=J\setminus\{n\}.
\]
We obtain a finite supercoercive convex function $f:\bR^m\to\bR$ such that $f^*$ is also finite and supercoercive, and
\[
    \bigcup_{u\in\bR^m}\sigma(S_f(u))=I_0,
    \qquad
    \bigcup_{x\in\bR^m}\sigma(S_{f^*}(x))=J_0.
\]
By Lemma~\ref{lem:epi},
\begin{equation}
\label{eq:epi-signatures}
    \sigma(\epi f)=I,
    \qquad
    \sigma(\epi f^*)=J.
\end{equation}

We homogenize the two epigraphs. Define the closed cone
\[
    K_f:=\cl\,\cone\{(x,t,1)\in\bR^m\times\bR\times\bR:t\ge f(x)\}.
\]
Since $f$ is finite and supercoercive, $K_f$ is a proper cone and
\[
    K_f\cap(\bR^{m+1}\times\{0\})
    =\{(0,t,0):t\ge0\}.
\]
Apart from this ray and the zero face $\{0\}$, every proper face of $K_f$ is of the form
\[
    \cone(F\times\{1\}),
\]
where $F$ is a proper face of $\epi f$. Consequently, every compact section of $K_f$ is an $n$-dimensional convex body with facial dimension signature $I$.

Moreover, a direct computation gives
\[
    K_f^\circ
    =\cl\,\cone\{(u,-1,s)\in\bR^m\times\bR\times\bR:s\le-f^*(u)\}.
\]
Because $f^*$ is finite and supercoercive, the same argument shows that every compact section of $K_f^\circ$ has facial dimension signature $J$.

Let $C\subseteq\bR^n$ be affinely isomorphic to a compact section of $K_f$, with $0\in\interior C$. Then $C^\circ$ is affinely
isomorphic to a compact section of $K_f^\circ$. Hence
\[
    \sigma(C)=I,
    \qquad
    \sigma(C^\circ)=J.
\]
Therefore, $C$ satisfies the desired properties, completing the construction.
\end{proof}

\begin{remark}
Theorem~\ref{thm:polar-realization} answers the second open question posed in \cite[Section~4.4]{RoshchinaTuncel2025}. The proof also explains why Fenchel conjugacy appears naturally: the prescribed signatures are first realized by the epigraphs of $f$ and $f^*$ and are then transferred to compact sections by homogenization.
\end{remark}

\section{A dimension obstruction under facial exposedness}
\label{sec:obstruction}

We now turn to a restriction on primal--polar facial dimension signatures under facial exposedness. The results complement the complete freedom in Theorem~\ref{thm:polar-realization}.

Let $C\subseteq\bR^n$ be a convex body with $0\in\interior C$. For a proper face $F\unlhd C$, its conjugate face is
\[
    F^\Delta:=\{u\in C^\circ:\langle u,x\rangle=1,\ \forall x\in F\}.
\]
The set $F^\Delta$ is an exposed face of $C^\circ$, and
\begin{equation}
\label{eq:dim-conjugate}
    \dim F+\dim F^\Delta\le n-1.
\end{equation}
If $C$ is facially exposed, then $F^{\Delta\Delta}=F$ for every proper face $F\unlhd C$. Consequently, for proper faces $F,G\unlhd C$,
\[
    F\subsetneq G\quad\Longrightarrow\quad G^\Delta\subsetneq F^\Delta.
\]
See \cite[Section~2.1]{Schneider2014} for these standard properties. We first record a simple consequence.

\begin{proposition}
\label{prop:positive-face-duality}
Let $C\subseteq\bR^n$ be a convex body with $0\in\interior C$, and suppose that both $C$ and $C^\circ$ are facially exposed. Then
\[
    \sigma(C)=\{0,n\}\quad\Longleftrightarrow\quad\sigma(C^\circ)=\{0,n\}.
\]
\end{proposition}

\begin{proof}
Suppose that $F\unlhd C$ is a positive-dimensional proper face. Choose $x\in\relbd F$, and let $G$ be the minimal face of $C$ containing $x$. Then $G\subsetneq F$, so $F^\Delta\subsetneq G^\Delta$. Hence $G^\Delta$ is a positive-dimensional proper face of $C^\circ$. The converse follows by symmetry.
\end{proof}

We next introduce the singular sets used in the dimension argument. For $x\in\bdry C$, write
\[
    x^\Delta:=\{u\in C^\circ:\langle u,x\rangle=1\},
\]
and, for $r\in\{0,\ldots,n-1\}$, define the $r$-singular set
\[
    \mathcal S_r(C):=\{x\in\bdry C:\dim x^\Delta\ge r\}.
\]
A theorem of Anderson and Klee \cite{AndersonKlee1952} implies that $\mathcal S_r(C)$ can be covered by countably many compact sets of finite $(n-r-1)$-dimensional Hausdorff measure. We write $\mathcal H^s$ for $s$-dimensional Hausdorff measure and $\dimH$ for Hausdorff dimension. We use the standard facts that if $K$ is an $a$-dimensional compact convex set with $a\ge1$, then
\begin{equation}
\label{eq:relbd-measure}
    \dimH(\relbd K)=a-1,\qquad 0<\mathcal H^{a-1}(\relbd K)<\infty,
\end{equation}
see \cite[Section~4.2]{Schneider2014}. And a set covered by countably many sets of finite $\mathcal H^s$-measure cannot contain an uncountable pairwise disjoint family of measurable subsets with positive $\mathcal H^s$-measure.

\begin{theorem}
\label{thm:dimension-obstruction}
Let $n\ge3$, and let $C\subseteq\bR^n$ be a convex body with $0\in\interior C$. Suppose that $C$ and $C^\circ$ both contain positive-dimensional proper faces and both $C$ and $C^\circ$ are facially exposed. Define
\[
    a:=\min\bigl(\sigma(C)\setminus\{0,n\}\bigr),
    \qquad
    b:=\min\bigl(\sigma(C^\circ)\setminus\{0,n\}\bigr).
\]
Then
\[
    a+b\le n-1.
\]
\end{theorem}

\begin{proof}
Fix $F\unlhd C$ with $\dim F=a$. By the minimality of $a$, any proper face of $F$ is a singleton. So, each $x\in\relbd F$ is a 0-dimensional face of $C$, and since $C$ is facially exposed,
\[
    F^\Delta\subsetneq x^\Delta.
\]
Thus, $x^\Delta$ is a positive-dimensional proper face of $C^\circ$, so $\dim x^\Delta\ge b$. Therefore,
\begin{equation}
\label{eq:rbFSb}
    \relbd F\subseteq\mathcal S_b(C).
\end{equation}
Combining with the Anderson--Klee theorem and \eqref{eq:relbd-measure}, we have
\[
    a-1=\dimH(\relbd F)\le\dimH\mathcal S_b(C)\le n-b-1,
\]
and hence $a+b\le n$.

Next, we will show that $a+b\neq n$. Suppose, to the contrary, that $a+b=n$. Since the hypotheses are symmetric in $C$ and $C^\circ$, we may assume that $a\ge2$. By \eqref{eq:dim-conjugate},
\[
    \dim F^\Delta\le n-1-a=b-1.
\]
If $\dim F^\Delta>0$, then the minimality of $b$ gives $\dim F^\Delta\ge b$, a contradiction. Hence
\[
    F^\Delta=\{x^*\}
\]
for some $x^*\in C^\circ$. Define
\[
    X:=\relbd F\setminus\mathcal S_{b+1}(C).
\]
The Anderson--Klee theorem gives $\dimH\mathcal S_{b+1}(C)\le a-2$, and hence
\[
    \mathcal H^{a-1}(\mathcal S_{b+1}(C))=0.
\]
Since $0<\mathcal H^{a-1}(\relbd F)<\infty$, the set $X$ is uncountable. Moreover, by the definition of $X$ and \eqref{eq:rbFSb},
\begin{equation}
\label{eq:dimxD}
    \dim x^\Delta=b,\qquad\forall x\in X.
\end{equation}

Let $x_1,x_2\in X$ and $x_1\neq x_2$. We claim that
\[
    x_1^\Delta\cap x_2^\Delta=F^\Delta=\{x^*\}.
\]
Indeed, $F^\Delta\subseteq x_1^\Delta\cap x_2^\Delta$ since $x_1,x_2\in X$. Conversely, if $z\in x_1^\Delta\cap x_2^\Delta$, then $F\cap z^\Delta$ is a face of $C$ containing the distinct points $x_1,x_2$. Thus it is positive-dimensional, so the minimality of $a$ yields
\[
    \dim(F\cap z^\Delta)\ge a.
\]
Since $F\cap z^\Delta\subseteq F$ and $\dim F=a$, we have $F\subseteq z^\Delta$. Hence $z\in F^\Delta=\{x^*\}$ and $x_1^\Delta\cap x_2^\Delta=\{x^*\}$. It follows that the sets
\[
    D_x:=\relbd(x^\Delta)\setminus\{x^*\},\qquad x\in X
\]
form an uncountable pairwise disjoint family. By \eqref{eq:dimxD},
\[
    \mathcal H^{b-1}(D_x)>0,\qquad\forall x\in X.
\]
Now fix $x\in X$ and $z\in D_x$. By \eqref{eq:dimxD} and the minimality of $b$, any proper face of $x^\Delta$ is a singleton. So, $\{z\}\subsetneq x^\Delta$ is a 0-dimensional face of $C^\circ$, and thus
\[
    (x^\Delta)^\Delta\subsetneq z^\Delta.
\]
Moreover, since $C$ is facially exposed and $\{x\}$ is a face of $C$, we have $(x^\Delta)^\Delta=\{x\}$. Thus $z^\Delta$ is a positive-dimensional proper face of $C$, and hence $\dim z^\Delta\ge a$. Therefore,
\[
    D_x\subseteq\mathcal S_a(C^\circ),\qquad\forall x\in X.
\]

By the Anderson--Klee theorem, $\mathcal S_a(C^\circ)$ can be covered by countably many compact sets of finite $\mathcal H^{b-1}$-measure. This contradicts the existence of the uncountable pairwise disjoint family $\{D_x:x\in X\}$ with $\mathcal H^{b-1}(D_x)>0$. Therefore $a+b\ne n$, and hence $a+b\le n-1$.
\end{proof}

\begin{corollary}
\label{cor:RT-open}
Let $n\ge3$. There is no convex body $C\subseteq\bR^n$ with $0\in\interior C$ such that both $C$ and $C^\circ$ are facially exposed and
\[
    \sigma(C)=\{0,n-1,n\}.
\]
\end{corollary}

\begin{proof}
Suppose that such a convex body $C$ exists. By Proposition~\ref{prop:positive-face-duality}, $C^\circ$ contains a positive-dimensional proper face. Thus
\[
    b:=\min\bigl(\sigma(C^\circ)\setminus\{0,n\}\bigr)
\]
is well-defined and $b\ge1$. Theorem~\ref{thm:dimension-obstruction} yields
\[
    n-1+b\le n-1,
\]
a contradiction.
\end{proof}

\begin{remark}
Corollary~\ref{cor:RT-open} gives a negative answer to the first open problem in \cite[Section~4.4]{RoshchinaTuncel2025}. This contrasts with Theorem~\ref{thm:polar-realization}, which shows that, in the absence of facial exposedness assumptions, the facial dimension signatures of a convex body and its polar can be prescribed independently. Thus, while polarity alone imposes no compatibility condition on the two signatures, simultaneous facial exposedness forces the restriction in Theorem~\ref{thm:dimension-obstruction}.
\end{remark}

\begin{acknowledgements}
The authors thank Vera Roshchina for helpful discussions on facial dimension signatures of convex sets.
\end{acknowledgements}

\section*{Statements and Declarations}
{\bf Conflict of interest}
The authors have no competing interests.

\bibliographystyle{spmpsci}
\bibliography{references}

\end{document}